\documentclass[12pt,oneside]{amsart} 
\usepackage{amssymb,amscd}
\usepackage{euscript}
 
      \theoremstyle{plain}
      \newtheorem{theorem}{Theorem}[section]
      \newtheorem{lemma}[theorem]{Lemma}
      \newtheorem{corollary}[theorem]{Corollary}
      \newtheorem{proposition}[theorem]{Proposition}
      \newtheorem{remark}[theorem]{Remark}
      
      \newtheorem{definition}[theorem]{Definition}        
          
\numberwithin{equation}{section}

      \makeatletter
      \def\@setcopyright{}
      \def\serieslogo@{}
      \makeatother

\def\A{\EuScript{A}} 
 
\def\E{\mathcal{E}}
\def\V{\mathcal{E}}

\def\M{{X}}

\def\R{\mathbb R}
\def\Rm{\mathbb R ^d}

\def\Z{\mathbb Z}
\def\N{\mathbb N}

\def\dist{\text{dist}}

\def\Id{\text{Id}}
\def\e{\epsilon}
\def\a{\alpha}
\def\ta{\tilde a}

\def\b{\beta}
\def\D{\Delta}

\def\la{\lambda}
\def\La{\Lambda}
\def\g{\gamma}

\def\ch{C^{1+\text{H\"older}}}

\def\r{\mathcal R}

\begin{document}

\author{Boris Kalinin$^{\ast}$ and Victoria Sadovskaya$^{\ast \ast}$}

\address{Department of Mathematics, The Pennsylvania State University, University Park, PA 16802, USA.}
\email{kalinin@psu.edu, sadovskaya@psu.edu}

\title[Lyapunov exponents of cocycles]
{Lyapunov exponents of cocycles over non-uniformly hyperbolic systems} 

\thanks{$^{\ast}$ Supported in part by Simons Foundation grant 426243}
\thanks{$^{\ast \ast}$ Supported in part by NSF grant DMS-1301693}


\begin{abstract} 

We consider linear cocycles over non-uniformly hyperbolic dynamical systems.
The base system is a diffeomorphism $f$ of a compact manifold $X$ preserving
a hyperbolic ergodic probability measure $\mu$. The cocycle $\A$ over $f$ is H\"older continuous
and takes values in $GL(d,\R)$ or, more generally, in the group of invertible bounded 
linear operators on a Banach space. For a $GL(d,\R)$-valued cocycle $\A$ we prove that the 
Lyapunov exponents of $\A$ with respect to $\mu$ can be approximated by the 
Lyapunov exponents of $\A$ with respect to measures on hyperbolic periodic orbits of $f$. 
In the infinite-dimensional setting one can define the upper and lower Lyapunov exponents 
of $\A$ with respect to $\mu$, but they cannot always be approximated by the exponents 
of $\A$ on periodic orbits. We prove that they can be approximated in terms of the 
norms of the return values of $\A$ on hyperbolic periodic orbits of $f$.

\end{abstract}

\maketitle 

\section{Introduction and statements of the results}

The theory of non-uniformly hyperbolic dynamical systems was pioneered by Ya. Pesin
in \cite{P1,P2} as a generalization of uniform hyperbolicity. It has become one of the central
areas in smooth dynamics with numerous applications, see \cite{BP,Po}. Periodic points play a major 
role in the study of both uniformly and non-uniformly hyperbolic systems. 
In the non-uniformly hyperbolic case, the existence of hyperbolic periodic orbits
and their relations to dynamical and ergodic properties of the system
were established by A. Katok in a seminal paper \cite{Kt80}. 
In fact, any hyperbolic invariant measure can be approximated  in weak* topology by
invariant measures supported on hyperbolic periodic points of the system \cite{BP}.
A further advance in this direction was obtained by Z. Wang and W. Sun who
showed  that Lyapunov exponents of any hyperbolic measure can be approximated 
by Lyapunov exponents of periodic points \cite{WS}. This does not follow from weak*
approximation as Lyapunov exponents in general do not depend continuously on the 
measure in weak* topology. 

The Lyapunov exponents above correspond to the derivative cocycle $Df$ of the base 
system $(X,f)$, which is a particular case of a linear cocycle, that is an automorphism 
of a vector bundle over $X$ that projects to $f$. Linear cocycles are the prime examples 
of non-commutative cocycles over dynamical systems. 
For uniformly hyperbolic systems, where every invariant measure is hyperbolic, 
a periodic approximation of Lyapunov exponents of linear  cocycles was 
established by the first author in \cite{K11}.  The results and techniques in \cite{K11} proved 
useful in various areas such as cohomology of non-commutative cocycles and 
the study of random and Markovian matrices and operators.  More recently, 
approximation results were obtained by the authors for cocycles with values in the group 
of invertible bounded linear operators on a Banach space \cite{KS17} and 
by L. Backes for semi-invertible matrix cocycles \cite{B}.

In this paper we extend the periodic approximation results to linear cocycles over 
 non-uniformly hyperbolic systems. 
 In the base we consider a diffeomorphism $f$ 
 of a compact manifold $X$  preserving a hyperbolic ergodic probability measure $\mu$. 
 The cocycles over $(X,f)$ will take values in $GL(d,\R)$ or, more generally, the group 
 $GL(V)$ of invertible bounded linear operators on a Banach space $V$. 
The space $L(V)$  of bounded linear operators on $V$  is a Banach space equipped 
with the operator norm $\|A\|=\sup \,\{ \|Av\| : \,v\in V , \;\|v\| \le 1\}.$ 
The open set $GL(V)$  of invertible elements in  $L(V)$ is a topological group  
and a complete metric space with respect  to the metric
$$
d (A, B) = \| A  - B \|  + \| A^{-1}  - B^{-1} \|.
$$

\begin{definition} Let $f$ be a homeomorphism of a compact metric space $X$ 
 and let 
$A$ be a function from $X$ to $GL(V)$. 
The $GL(V)${\em -valued  cocycle over $f$ generated by }$A$ 
is the map $\A:\,\M \times \Z \,\to GL(V)$ defined  by $\A(x,0)=\Id$ and for $n\in \N$, 
 $$
\A(x,n)=\A_x^n = 
A(f^{n-1} x)\circ \cdots \circ A(x), \;\quad
\A(x,-n)=\A_x^{-n}= (\A_{f^{-n} x}^n)^{-1}.
$$
 In the finite-dimensional case of $V=\R^d$ we will call $\A$ a $GL(d,\R)$-valued cocycle.
\vskip.05cm
We say that the cocycle $\A$ is  {\em $\a$-H\"older}, $0<\a \le 1$, if there exists $M>0$ such that
 \begin{equation} \label{holder}
 d (\A_x, \A_y) \le  \, M \dist (x,y)^\alpha \quad\text{for all  }x,y \in X. 
\end{equation}
\end{definition}

Clearly, $\A$ satisfies the cocycle equation $\A^{n+k}_x= \A^n_{f^k x} \circ \A^k_x$.
Since $X$ is compact, the H\"older condition \eqref{holder} is equivalent to 
$\| \A _x - \A_y \| \le  \, M' \dist (x,y)^\alpha$ for all $x,y \in X$.
H\"older continuity of a cocycle is natural in our setting as the lowest regularity 
allowing development of a meaningful theory beyond the measurable case. It covers
the case of the derivative cocycles of $\ch$ diffeomorphisms and their restrictions
to H\"older continuous sub-bundles of $TX$, which play an important role in hyperbolic systems. 
\vskip.2cm

Any $GL(V)$-valued   cocycle $\A$ can be viewed as an automorphism of the trivial vector 
bundle $\E=X\times V$, $\A (x,v)= (fx, \A_x (v))$. More generally, we can consider a linear
cocycle $\A$, i.e. an automorphism of any vector bundle $\E$ over $X$ that projects to  $f$.
This setting covers the case of the derivative cocycle $Df$ of a diffeomorphism $f$ of $\M$ with 
nontrivial tangent bundle. For any measure $\mu$ on $\M$, any vector bundle $\E$ over $\M$
is trivial on a set of full measure \cite[Proposition 2.1.2]{BP} and hence any linear cocycle $\A$ can be 
viewed  as a  $GL(V)$-valued cocycle on a set of full measure.

First we consider the finite dimensional case where the Lyapunov exponents and Lyapunov 
decomposition for $\A$ with respect to an ergodic $f$-invariant measure $\mu$ are given 
by Oseledets Multiplicative Ergodic Theorem.
We note that both are defined $\mu$-a.e. and  depend on the choice of $\mu$.  
\vskip.2cm

\noindent{\bf Oseledets Multiplicative Ergodic Theorem.}\, \cite{O,BP} \,
{\it Let $f$ be an invertible ergodic measure-preserving transformation of a 
Lebesgue probability space $(X,\mu)$. Let $\A$ be a measurable $GL(d,\R)$-valued cocycle over $f$ 
satisfying 
$\,\log\|\A_x\|\in L^1(X,\mu)$ and $\log\|\A_x^{-1}\|\in L^1(X,\mu)$.
Then there exist numbers $\la_1 < \dots < \la_{m}$, an $f$-invariant 
set $\La$ with $\mu (\La)=1$, and an $\A$-invariant Lyapunov decomposition 
$$
\R^d=\E_x  =\E^{1}_x\oplus\dots\oplus \E^{m}_x\;\; \text{ for }\,x\in \La \text{ such that }
$$
 \begin{itemize}
\item[(i)] $ \underset{n\to{\pm \infty}}{\lim} n^{-1} \log\| \A^n_x\, v \|=  \la_i\, $  
for any $i=1,\dots ,m$ and any $\,0 \not= v\in \E^{i}_x$, and
 \vskip.1cm
\item[(ii)]  $ \underset{n\to{\pm \infty}}{\lim} n^{-1}  \log |\det \A^n_x |
= \sum_{i=1}^{m}d_i  \la_i $,\,\,
where $\,d_i=\dim \E^{i}_x$. 
\end{itemize}
}

\begin{definition} The numbers $\la_1, \dots, \la_{m}$ are called the {\it Lyapunov exponents}\,
of $\A$ with respect to $\mu$ and the integers $d_1, \dots , d_m$ are called their multiplicities.
\end{definition}

\begin{definition}  Let $\mu$ be an ergodic invariant  Borel probability measure for
a diffeomorphism $f$  of a compact manifold $X$. The measure is called {\em hyperbolic}
if all the Lyapunov exponents of the derivative cocycle $Df$ with respect to $\mu$ are non-zero.
\end{definition}

By Lyapunov exponents of $\A$ at a periodic point $p=f^kp$ we mean the Lyapunov exponents
of $\A$ with respect to the invariant measure $\mu_p$ on the orbit of $p$. They equal $(1/k)$ of 
the logarithms of the absolute values of the eigenvalues of $\A^k_p$. A periodic point $p$ is
called {\em hyperbolic} if $Df$ has no zero exponents at $p$, that is $D_pf^k$ has no 
eigenvalues of absolute value $1$. 

The following theorem extends the periodic approximation results in \cite{WS} and  \cite{K11}
to linear cocycles over  non-uniformly hyperbolic systems.

\begin{theorem} \label{main fin} 
Let $f$ be a $C^{1+\text{H\"older}}$ diffeomorphism  of a compact manifold $X$,
 let $\mu$ be a hyperbolic ergodic $f$-invariant Borel probability measure on $X$, 
 and let $\A$ be a $GL(d,\R)$-valued H\"older continuous cocycle over $f$.

Then the Lyapunov exponents
 $\,\la_1 \le \dots \le \la_d\,$ of $\A$ with respect to $\mu$,\, listed with multiplicities, 
 can be approximated by the Lyapunov exponents of $\A$ at periodic points. 
More precisely, for each $\e >0$ there exists a hyperbolic periodic point $p \in \M$ 
for which the Lyapunov exponents $\la_1^{(p)} \le \dots \le \la_d^{(p)}$ of $\A$ satisfy
\begin{equation} \label{main eq d}
|\la_i-\la_i^{(p)}|<\e\, \quad\text{for }\,i=1, \dots , d.
\end{equation}
\end{theorem}
\vskip.3cm

The largest and smallest
Lyapunov exponents $\lambda_+(\A,\mu)=\lambda_m$ and $\lambda_-(\A,\mu)=\lambda_1$ 
can be expressed as follows:
\begin{equation} \label{exponents}
\begin{aligned}
&\lambda_+(\A,\mu)=  \lim_{n \to \infty} n^{-1}  \log \| \A_x ^n \| 
\quad \text{for } \mu \text {-a.e. } x\in \M  , \\
&\lambda_-(\A,\mu)=  \lim_{n \to \infty}  n^{-1}  \log \| (\A_x ^n)^{-1} \|^{-1} 
\quad \text{for } \mu \text {-a.e. }x\in \M .
\end{aligned}
\end{equation}

While there is no Multiplicative Ergodic Theorem  in the infinite-dimensional case in general, 
the upper and lower Lyapunov exponents $\la_+$ and $\la_-$ of $\A$ can still be 
defined by \eqref{exponents}, see Section \ref{subadditive}. 
For the invariant measure $\mu_p$ on the orbit of  $p=f^kp$ we have
 $$
 \lambda_+(\A,\mu_p) = k^{-1} \log \,(\text{spectral radius of }\A_p^k)
 \,\le\, k^{-1} \log \| \A_p ^k \| .
 $$
In the infinite-dimensional setting, 
it is not always possible to approximate $\,\lambda_+(\A,\mu)$ by $\lambda_+(\A,\mu_p)$,
even for cocycles over uniformly hyperbolic systems \cite[Proposition 1.5]{KS17}. 
However, an approximation of $\,\lambda_+(\A,\mu)$ by $k^{-1}  \log \| \A_p ^k \|$ was
obtained in \cite{KS17} for cocycles over uniformly hyperbolic systems.
The next theorem establishes such an approximation in the non-uniformly 
hyperbolic setting.

\begin{theorem} \label{main inf} 
Let $f$ be a $\ch$ diffeomorphism  of compact manifold $X$,
 let $\mu$ be a hyperbolic ergodic $f$-invariant Borel probability measure on $X$, 
 and let $\A$ be a H\"older continuous $GL(V)$-valued cocycle over $f$.

Then  for each  $\e>0$ there exists a hyperbolic periodic point $p=f^kp$ in $X$
such that 
\begin{equation} \label{main eq}
\left| \,\lambda_+(\A,\mu)-  k^{-1} \log \| \A_p ^k \| \, \right|<\e \;\text{ and }\;
\left| \, \la_- (\A,\mu)- k^{-1} \log \| (\A_p ^k)^{-1} \| ^{-1} \, \right|<\e.
\end{equation}
Moreover, for any $N\in \N$ there exists such $p=f^kp$ with $k>N$.

\end{theorem}

The proof of the finite dimensional approximation in Theorem \ref{main fin} relies on Multiplicative Ergodic Theorem, which yields that the cocycle has
finitely many Lyapunov exponents and, in particular, the largest one is isolated. 
As this may not be the case in infinite dimensional setting even for a single operator,
in Theorem \ref{main inf} we use a different approach which  relies on
norm estimates. In particular we use   a suitable version of Lyapunov norm and 
results on subadditive cocycles \cite{KM}.


\begin{remark} \label{several}
Theorems \ref{main fin} and \ref{main inf} can be strengthened to conclude the 
existence of a hyperbolic periodic point $p=f^kp$ which gives simultaneous 
approximation as in \eqref{main eq d} and \eqref{main eq} for finitely many  
cocycles $\A^{(j)}$, $i=1, \dots , m$, over $f$ with values in $GL(V_i)$. We describe the 
 modifications for this case in the proofs.
\end{remark}

\begin{remark}
Theorems  \ref{main fin} and \ref{main inf} hold if we 
replace $X\times V$ by a H\"older continuous vector bundle $\V$ over 
$X$ with fiber $V$ and the cocycle $\A$ by an automorphism $\A: \V \to \V$ 
covering $f$. 
This setting is described in detail in Section 2.2 of \,\cite{KS13} and the proofs 
work without any significant modifications.  
\end{remark}


\section{Preliminaries}
For a $GL(V)$-valued cocycle $\A$ over $(X,f)$ we consider the trivial bundle
$\V=X\times V$ and view $\A^n_x$ as a fiber map from $\V_x$ to $\V_{f^nx}$. 
This makes notations and 
arguments more intuitive and the  extension to non-trivial bundles more transparent.
In fact, all our arguments are written for the bundle setting, except we 
sometimes identify fibers $\V_{x}$ and $\V_{y}$ at nearby points $x$ and $y$. 
This is automatic for a trivial bundle, and a detailed description of a suitable 
identification for a non-trivial bundle is given in \cite[Section 2.2]{KS13}.

\subsection{Lyapunov metric}\label{SLyapunovmetric} 

We consider a $GL(d,\R)$-valued cocycle $\A$ over $(X,f,\mu)$ as in the Oseledets 
Multiplicative Ergodic Theorem
and denote the standard scalar product in $\Rm$ by $\langle \cdot,\cdot \rangle$. 
We fix  $\e >0$ and for any  point $x\in \La$  define the {\it Lyapunov 
scalar product}  $\langle \cdot,\cdot \rangle_{x,\e}$ on $\Rm$ as follows.
\vskip.2cm

For $u\in \E_x^i,\,\,v\in \E_x^j,\,\,i\neq j$,  we set  $\;\langle u,v\rangle_{x,\e} =0$.

\vskip.2cm
For $u,v\in \E_x^i$, $\,i=1,\dots,m$,  we define
$
\;\langle u,v \rangle_{x,\e}  =d \,\underset{n\in\Z}{\sum }\,\langle\, \A_x^n\, u,\,\A_x^n\,v \,\rangle\, e^{-2\la_i n -\e |n|}. 
$ 
\vskip.2cm
\noindent The series converges exponentially for any $x \in \La$. The constant $d$ in the formula 
allows a more convenient comparison with the standard scalar product. 
The norm generated by this scalar product is called the {\em Lyapunov norm}
and is denoted by $\|.\|_{x,\e}$. When $\e$ is fixed we will denote the scalar 
product by $\langle \cdot,\cdot \rangle_x$ and the  norm by $\|.\|_x$.

\vskip.1cm

We summarize the main properties of the Lyapunov scalar product 
and norm, see  \cite[Sections 3.5.1-3.5.3]{BP}  for more details. A direct calculation 
shows \cite[Theorem 3.5.5]{BP} that for any $x\in \La$ and any $u\in \E_x^i$,
\begin{equation}  \label{estAEi}
e^{n \la_i -\e |n|} \cdot \|u\|_{x,\e} \le 
\| \A_x^n \, u\|_{f^n x,\e} \le
e^{n \la_i+\e |n|}\cdot \|u\|_{x,\e} \quad \text{for all } n \in \Z,
\end{equation}
\begin{equation}  \label{estAnorm}
e^{n \la_+ -\e n} \le \| \A_x^n \|_{f^n x \leftarrow x} \le e^{n \la_+ +\e n}
\quad \text{for all }  n \in \N,
\end{equation}
where $\la_+ = \la _m$ is the largest Lyapunov exponent and 
$\| . \|_{f^n x \leftarrow x}$ is the operator norm with respect to the
Lyapunov norms defined as
$$
\| A \| _{y \leftarrow x}=\sup \,\{ \| Au \|_{y,\e} \cdot \| u\|_{x,\e} ^{-1} : \; 0 \not= u \in \Rm \}.
$$

The Lyapunov scalar product and norm are defined only on $\La$ and, in general, depend only 
measurably on the point even if the cocycle is H\"older, so comparison with the standard
norm is important. The lower bound follows easily from the definition: $\|u\|_{x,\e}\ge \|u\|$. 
An upper bound is not uniform, but can be chosen to change slowly along the orbits 
\cite[Proposition 3.5.8]{BP}: there exists a measurable function $K_\e (x)$  on  $\La$ 
such that 
\begin{equation}  \label{estLnorm} 
\| u \| \le  \| u \|_{x,\e} \le K_\e(x) \|u\| \qquad \text{for all } x \in \La
\text{ and }  u \in \Rm, \quad \text{and} 
\end{equation}
\begin{equation}  \label{estK}  
 K_\e(x) e^{-\e |n|}  \le K_\e(f^n x) \le  K_\e(x) e^{\e |n|}  \qquad 
\text{for all }  x \in \La \text{ and }  n \in \Z.
\end{equation}

For any matrix $A$ and any points $x,y\in \La$ inequalities 
\eqref{estLnorm} and \eqref{estK} yield 
\begin{equation}  \label{estMnorm}
K_\e (x) ^{-1} \| A \| \le \| A \|_{y \leftarrow x} \le K_\e (y) \| A \| \, .
\end{equation}

For any $\ell>1$ we define the sets 
\begin{equation}  \label{Pset}
 \La_{\e,\ell}= \{x \in \La : \; \; K_\e(x) \le \ell \}.
\end{equation}
and note that $\mu (\La_{\e,\ell}) \to 1$ as $\ell \to \infty$.
Without loss of generality we can assume that the set $\La_{\e,\ell} $ is compact
and that Lyapunov splitting and Lyapunov scalar product are continuous on
$\La_{\e,\ell} $. Indeed, by Luzin theorem we can always find a subset of $\La_{\e,\ell} $ 
satisfying these properties with arbitrarily small loss of measure.


\subsection{Pesin sets and Closing lemma} \label{Closing lemma}

Let $f$ be a diffeomorphism of a compact manifold  $X$ and $\mu$ be an
ergodic $f$-invariant  Borel probability measure. We apply the Multiplicative Ergodic Theorem
and construct the Lyapunov metric as  above for the derivative cocycle $\A_x=D_x f$.
For this cocycle, we will denote the corresponding set $\La_{\e,\ell}$ defined 
in the previous section by $\r_{\e,\ell}$, which is often called a Pesin set.

Suppose now that  the measure $\mu$ is hyperbolic, i.e. all Lyapunov exponents of the 
derivative cocycle $Df$ with respect to $\mu$ are non-zero. 
We assume that there are both positive and negative 
such exponents. Otherwise $\mu$ is an atomic measure on a single periodic orbit 
\cite[Lemma 15.4.2]{BP}, in which case our results are trivial. 
We denote by $\chi >0$ the smallest absolute
value for these exponents. We will fix $\e>0$ sufficiently small compared to $\chi$
and $\ell \in \N$  large enough so that the corresponding Pesin set  $\r _{\e,\ell}$ has 
positive measure. We will apply the following closing lemma.
It does not use the splitting into individual Lyapunov sub-bundles $\E^i$, only the stable/unstable sub-ones, which are the sums of all  Lyapunov sub-bundles corresponding to negative/positive Lyapunov exponents, respectively. Consequently, a cruder version of a Pesin set can be used instead of $\r _{\e,\ell}$.
 
\begin{lemma} [Closing Lemma]  \cite{Kt80}, \cite[Lemma 15.1.2]{BP} \label{closing} 
Let $f : \M \to \M$ be a $\ch$ diffeomorphism preserving a hyperbolic Borel probability 
measure $\mu$ and let be $\chi >0$ the smallest absolute value of its Lyapunov  exponents.
 Then for any sufficiently large $\ell \in \N$, any sufficiently small $\e>0$, and any $\delta > 0$
 there exist $\g =\g (\e,\ell) \in (\e, \chi -2\e)$ and  $\b = \b(\delta,\e,\ell) > 0$  such that if 
 $$
 x \in \r _{\e,\ell} ,\;\;f^kx \in \r _{\e,\ell} \;\text{  and }\,\dist(x,f^kx) < \b \;\text{ for some $k \in \N$}, 
 $$ 
 then there exists 
  a hyperbolic periodic point $p=f^kp$ such that
\begin{equation}  \label{d-close-traj}
\dist (f^i x, f^i p) \le  \delta \, e^{ -\g\, \min \{\,i,\,k-i \,\} }
\quad\text{for every }i=0, \dots , k.
\end{equation}
\end{lemma}

While the lemma is usually stated with a constant  on the right hand side of \eqref{d-close-traj},
this constant can be absorbed using the choice of $\beta$. We will not use hyperbolicity of the
 periodic point $p$ in the proof. 
 In fact, for sufficiently large $k$ the hyperbolicity of $p$ can be recovered by applying our argument 
 to the derivative cocycle $Df$.


\section{Proof of Theorem \ref{main fin}}

Let $\la_1 < \dots < \la_m$ be the Lyapunov exponents of $\A$  with respect to  $\mu$,
listed without multiplicities. 
We will denote the largest exponent $\la_m$ by $\la$ and second largest $\la_{m-1}$ by $\la'$.
Similarly, for any periodic point $p$ we denote by $\la^{(p)}$ the largest Lyapunov 
exponent of $\A$ at $p$. 

We fix $\e'>0$ sufficiently small compared to $\chi$ and $\ell' \in \N$  large enough 
so that the corresponding Pesin set  $\r _{\e',\ell'}$ for the derivative cocycle $Df$
of  the base system has positive measure. We apply the Closing Lemma \ref{closing} and get 
$\g =\g (\e',\ell') >0$.

If $\la$ is not the only Lyapunov exponent 
of $\A$ with respect to $\mu$, we define
\begin{equation} \label{e0}
\e _0 = \min \,\{\,  \a \g , \;(\la - \la')/4,\; \e'\, \},
\end{equation} 
and otherwise we set $\e _0 =  \min \,\{\,  \a \g ,\, \e' \,\}$.

We fix $0<\e<\e_0$ and consider the sets $\La_{\ell, \e}$ for the cocycle $\A$. 
We denote 
$$
P=\La _{\ell, \e} \cap \r_{\ell', \e'}.
$$
and fix $\ell$  sufficiently  large so that $\mu (P)>0$.

We take a point $x \in P$ which is in the support of $\mu$ restricted to $P$.
Then for any $\beta>0$ we have $\mu(P \cap B_{\beta/2}(x))>0$, where
$B_{\beta/2}(x)$ is the open ball of radius $\beta/2$ centered at $x$.
By Poincare recurrence there are infinitely many $k\in \N$ such that $f^k x \in P \cap B_{\beta/2}(x)$.
For any such $k$ we have: $x, f^kx\in P$ and $\dist (x, f^k x) <\beta$.
Taking $\b=\b(\delta,\e',\ell') > 0$ from the Closing Lemma \ref{closing} we obtain 
a hyperbolic periodic point satisfying \eqref{d-close-traj}. 
We can assume that $\b \le \delta$, so that when $\delta$ is small so is $\beta$. 

We conclude that for each $\delta >0$ there exist arbitrarily large $k$ such that
$x, f^kx\in P$ and there is a hyperbolic periodic point $p=f^kp$ satisfying \eqref{d-close-traj}.
Now we show that for such a point $p$ with a sufficiently large $k$ and a sufficiently small
$\delta$  we have
\begin{equation} \label{largestexp}
|\la-\la^{(p)}| \le 3\e. 
\end{equation} 
To estimate $\la^{(p)}$ from above we use the fact \cite[Lemma 3.1]{K11} that for such a point $p$
\begin{equation}  \label{d-close-normO}
\| \A_p^k \| \le \ell \, e^{c \,  \ell \delta^\a} e^{k (\la+\e) },
\end{equation}
where the constant $c$ depends only on $\A$ and on the number $(\a \g - \e)$, which
also follows from Lemma \ref{mainest} below.
Since we chose $\e<\a\g$ we obtain
$$
\la^{(p)} \le k^{-1} \log \| \A_p^k \|  \le \la + \e + k^{-1} \log ( \ell \, e^{c \,  \ell \delta^\a}) \le \la + 2\e
$$ 
provided that $\delta <1$ and  $k$ is large enough compared to $\ell$.
\vskip.1cm 

Now we  estimate $\la^{(p)}$ from below.
We denote $x_i=f^i x$ and $p_i=f^i p$,
and  we write $\| . \| _i$ for the Lyapunov norm at $x_i$.
Since the Lyapunov norm may not exist at points $p_i$ we will use the Lyapunov 
norms at the  corresponding points  $x_i$ for the estimates. 
For each $i$ we have the orthogonal splitting $\Rm =  \E_i ^{(1)}\oplus\E_i^{(2)}$, 
where $\E_i ^{(1)}=\E^m_{x_i}$ is the Lyapunov space at $x_i$ corresponding to the largest Lyapunov 
exponent $\la=\la_m$, and $\E_i^{(2)}$ is the direct sum of  the other Lyapunov spaces at $x_i$. 
We will assume that $\la$ is {\em not} the only Lyapunov exponent 
of $\A$, as otherwise $\E_i^{(2)}=\{ 0 \}$  and the argument is simpler.
 For a vector $u \in \Rm$ we write $u = u_1 + u_2$, where $u_1\in \E_i^{(1)}$ 
 and $u_2 \in \E_i^{(2)}$.
 \vskip.1cm
 
We take $\theta=e^{\la' -\la+4\e}<1$ by the choice of $\e$. For  $i=0, \dots , k $ we consider the cones
$$
K_i = \{ u \in \Rm : \,\| u_2 \| _i \le  \| u_1 \| _i\} \quad \text{and} \quad
K_i^\theta = \{ u \in \Rm : \,\| u_2 \| _i \le \theta\, \| u_1 \| _i \,\}.
$$
Now we show that there exist $\delta_0 > 0$ such that for all $0<\delta<\delta_0$
and all  $i=0, \dots , k-1$
\begin{equation} \label{Ki}
\A_{p_i} (K_i ) \subset K_{i+1}^\theta \quad  \text{and} \quad
\| \left( \A_{p_i} u \right) _1 \|_{i+1} \ge e^{\la -2\e} \| u_1 \| _i \; \text{ for each $u \in K_i$}.
\end{equation}
We fix $0 \le i < k$ and a vector $u \in K_i$. Denoting  $\A_{x_i} u =v = v_1 + v_2$ we get 
 by \eqref{estAEi} 
\begin{equation} \label{v}
 e^{\la - \e} \| u_1 \| _i \le \| v_1 \| _{i+1}   \le e^{\la + \e} \| u \| _i \quad \text{and} \quad
  \| v_2 \| _{i+1}  \le e^{\la' + \e} \| u_2 \| _i \, .
\end{equation}   
We  write
$\A_{p_i} =(\Id + \D_i ) \, \A_{x_i},$ where 
$\D_i = \A_{p_i} (\A_{x_i})^{-1} - \Id =  (\A_{p_i} - \A_{x_i}) (\A_{x_i})^{-1} .$
\vskip.1cm

Since both $x_0$ and $x_k$ are in $\La_\ell$ and $\dist (x_i, p_i) \le \delta e^{-\g \min \{i,\, k-i\} }$,
using \eqref{estMnorm}
 we can estimate 
\begin{equation} \label{DE}
\begin{aligned}
& \| \D_i \|_{x_{i+1} \leftarrow x_{i+1}}  \,\le\, 
 K (x_{i+1}) \,\| \D_i \|  \le   K (x_{i+1})\,\| \A_{p_i} - \A_{x_i} \| \cdot \| (\A_{x_i})^{-1} \| \,\le \\
& \le  K (x_{i+1})\cdot c_1\, \dist (x_i, p_i) ^ \a   \, \le  \,
 \ell e^{\e \min \{i+1,\, k-i-1\} } \cdot c_1 \delta^ \a e^{ -\a\g \min \{i, \,k-i\} } \le \\
 &\le \ell c_1  \delta^\a e^\e e^{(-\g \a + \e) \min \{i, \,k-i\} } \le c_2 \ell \delta^\a
 \quad\text{since $-\g \a + \e <0$.}
  \end{aligned}
\end{equation} 
 Since $\| u \| _i  \le \sqrt{2}\, \| u_1 \| _i$ we conclude using \eqref{v} that
\begin{equation} \label{Dv}
 \| \D_i \,v \| _{i+1} \le  \| \D_i \|_{x_{i+1} \leftarrow x_{i+1}} \| v \| _{i+1} \le 
  c_2 \ell \delta^\a  e^{\la + \e} \| u \| _i  \le c_3 \ell \, \delta^\a   \| u_1 \| _i.
 \end{equation} 
Setting $w=\A_{p_i} u =(\Id + \D_i ) \A_{x_i}u= (\Id + \D_i )v\,$ we observe that  
\begin{equation} \label{wv}
w_1 = v_1 + (\D_i v)_1 \quad \text{ and } \quad w_2 = v_2 + (\D_i v)_2
\end{equation}
 and hence using \eqref{v} and \eqref{Dv} we obtain that for small enough $\delta$
$$
 \| w_1 \| _{i+1} \ge \| v_1 \| _{i+1} -  \| \D_i v \| _{i+1}  \ge 
e^{\la - \e} \| u_1 \| _i - c_3 \ell \, \delta^\a   \| u_1 \| _i 
\ge e^{\la - 2\e}   \| u_1 \| _i \, ,
$$
which gives the  inequality in \eqref{Ki}. Similarly, using  $\| u_2 \| _i \le \| u_1 \| _i$, we get
$$
\begin{aligned}
 \| w_2 \| _{i+1} & \le \| v_2 \| _{i+1} +  \| \D_i v \| _{i+1} \le
   e^{\la' + \e} \| u_2 \| _i + c_3 \ell \, \delta^\a   \| u_1 \| _i \le \\
  &\le  (e^{\la' + \e} + c_3 \ell \, \delta^\a )  \| u_1 \| _i \le
   e^{\la' + 2\e} \| u_1 \| _i
\end{aligned}
$$
for all sufficiently small $\delta$.
Finally, if $u\ne0$  we get that
$$
 \| w_2 \| _{i+1} \,/ \,\| w_1 \| _{i+1} \le  e^{\la' +2\e} / e^{\la - 2\e} 
 =e^{\la' -\la+4\e}= \theta.
$$
This shows that $w \in K_{i+1}^\theta$ and the inclusion
$\A_{p_i} (K_i ) \subset K_{i+1}^\theta$ in \eqref{Ki} follows.
\vskip.2cm

We conclude that  \eqref{Ki} holds for each $i=0, \dots , k-1 $ and hence 
$\A_p^k (K_0 ) \subset K_{k}^\theta$. 
Since $\La_\ell$ is chosen compact and so that the Lyapunov splitting and 
Lyapunov metric are continuous on it, the cones $K^\theta_0$ and $K^\theta_k$ are close
if $\dist(x,f^k x)<\b$ is small. Thus we can ensure that
$K_{k}^\theta  \subset K_0$ if $\b$ small enough and hence 
$\A_p^k (K_0) \subset K_{0}$. Finally, using the  inequality in \eqref{Ki}
for each $i=0, \dots , k-1 $ we obtain that for any $u \in K_0$  
$$
\| \A_p^k\, u \| _{k} \ge \| (\A_p^k\, u)_1 \| _{k} \ge  e^{k(\la - 2\e)} \| u_1 \| _0 
\ge e^{k(\la - 2\e)} \| u \| _0 /\sqrt{2} \ge  e^{k(\la - 2\e)} \| u \| _k /2
$$
since Lyapunov norms at $x$ and $f^k x$ are close if $\delta$ and hence $\b$ is small enough.
Since $\A_p^k\, u \in K_0$ for any $u \in K_0$, we can iterate $\A_p^k$ and
use the inequality above to estimate $\la^{(p)}$ from below by the exponent of 
any $u \in K_0$:
$$
\la^{(p)} \ge \, \lim _{n \to \infty} \,(nk)^{-1} \log \| \A_p^{nk} \,u \|_{k} \ge k^{-1} \cdot
\lim _{n \to \infty} n^{-1} \log  \left( (e^{k(\la - 2\e)}/2)^n \,\| u \| _k \right) \ge
$$
$$
\ge  k^{-1} \left[k(\la - 2\e) - \log 2 \right] + k^{-1} \lim _{n \to \infty}n^{-1}\log \| u \| _k
\ge (\la - 2\e) -  k^{-1}\log 2 \ge \la - 3\e
$$
provided that $k$ is large enough. This completes the proof of the approximation
of the largest exponent \eqref{largestexp}. 

\vskip.1cm
To approximate all Lyapunov
exponents of $\A$ we consider cocycles $\wedge ^i \, \A$ induced by 
$\A$ on the $i$-fold exterior powers $\wedge ^i \, \Rm$, for $i= 1, \dots , d$.
The largest Lyapunov exponent of $\wedge ^i \, \A$ is 
$(\la _d + \dots + \la_{d-i+1})$, where $\la _1 \le \dots \le \la_{d}$ are the
Lyapunov exponents of $\A$ {\it listed with multiplicities.} If a periodic point $p=f^kp$
satisfies 
$$
|(\la _d + \dots + \la_{d-i+1})-(\la^{(p)} _d + \dots + \la^{(p)}_{d-i+1})| \le 3\e
\quad\text{for $i= 1, \dots , d$,}
$$
 then we obtain the approximation  $|\la_i - \la^{(p)}_i| \le 3d\e$ 
for all $i= 1, \dots , d$,  completing the proof of Theorem \ref{main fin}. Such a periodic 
point exists since we can take a set 
$$
P= \r_{\e', \ell'} \cap \La^1_{\e, \ell} \cap \dots \cap \La^d_{\e, \ell}$$ 
with $\mu (P) >0$, where $\La^i_\ell$ are the corresponding sets for all cocycles 
$\wedge ^i \, \A$, $i= 1, \dots , d$. Then the previous argument applies 
yielding \eqref{largestexp} for each $\wedge ^i \, \A$.

A similar argument shows that one can obtain a simultaneous approximation of all
Lyapunov exponents for several cocycles.


\section{Subadditive cocycles and infinite-dimensional  Lyapunov norm}

\subsection{Subadditive cocycles and their exponents} \label{subadditive}
A {\em subadditive cocycle} over a dynamical system $(X,f)$
is a sequence of functions $a_n:X\to \R\,$ such that 
$$
a_{n+k}(x)\le a_k(x) + a_n(f^kx) \quad \text{for all }x\in X \text{ and }k,n\in \N.
$$
For any ergodic measure-preserving transformation $f$ of a probability space 
$(X,\mu)$ and any subadditive cocycle over $f$ with integrable $a_n$, 
 the Subadditive Ergodic Theorem yields that for $\mu$ almost all  $x$ 
$$
 \lim_{n \to \infty} \frac {a_n(x)}{n}= \lim_{n \to \infty} \frac {a_n(\mu)}{n}=
\inf_{n \in \N} \frac{a_n (\mu)}{n} =: \nu (a,\mu) ,\; 
\text{ where }\; a_n(\mu)= \int_\M a_n(x) d\mu. 
$$
The limit $\nu (a,\mu) \ge -\infty$ is called the {\em exponent}\, of the cocycle $a_n$
with respect to $\mu$. 

For a $GL(V)$-valued cocycle $\A$ over $(X,f)$ 
it is easy to see that 
\begin{equation}\label{an}
a_n(x)=\log \|\A_x^n\| \quad\text{and}\quad
\ta_n(x)=\log \|(\A_x^n)^{-1}\|.
\end{equation}
are subadditive cocycles over $f$.   
For these continuous cocycles the Subadditive Ergodic Theorem gives the existence 
of the limits in \eqref{exponents}: for  $\mu$ almost all  $x$
\begin{equation}\label{la}
\la_+ (\A,\mu)=\lim_{n \to \infty} \,n^{-1} \log \|\A_x^n\| =\lim_{n \to \infty} n^{-1}{a_n(x)}= 
\nu(a,\mu), 
\end{equation}
\begin{equation}\label{-chi}
\quad -\la_- (\A,\mu) =\lim_{n \to \infty}\, n^{-1} \log \| (\A_x^n)^{-1} \| 
= \lim_{n \to \infty} n^{-1}\,{\ta_n(x)}= \nu(\ta,\mu) . 
\end{equation}
It is easy to see that $\la_- \le \la_+$ and both are finite. Also, for  $\mu$ almost all  $x$
\begin{equation}\label{-chi2}
-\la_- (\A,\mu) =\lim_{n \to \infty}\, n^{-1} \log \| \A_x^{-n} \|\hskip4.5cm
\end{equation}
since the integrals of $\log \| \A_x^{-n} \|$ and $\log \| (\A_x^n)^{-1} \|$ are equal.  We denote 
\begin{equation} \label{Lambda def}
\La =\La(\A,\mu)=\{ x\in X : \text{ equations \eqref{la}, \eqref{-chi}, and \eqref{-chi2} hold}\,\},
\end{equation}
which implies that both equalities in \eqref{exponents} hold for $x \in \La$. Clearly, $\mu(\La)=1$.
\vskip.1cm

We will also use the following more detailed result on the behavior of subadditive cocycles 
established  by A. Karlsson and G. Margulis.

\begin{proposition}\cite[Proposition 4.2]{KM} \label{GK} 
Let $a_n(x)$ be an integrable subadditive cocycle with exponent $\la>-\infty$ 
over an ergodic measure-preserving system $(X,f,\mu)$.
Then there exists a  set  $E \subset X$ with $\mu(E)=1$ such that 
for each $x\in E$ and each $\e >0$ there exists an integer $L=L(x,\e)$ so that
the set $S=S(x,\e, L)$ of integers $n$ satisfying 
\begin{equation}\label{good n}
  a_n(x)-a_{n-i} (f^ix) \ge (\la-\e)i \quad\text{for all $i$ with }\,L\le i \le n
\end{equation}
is infinite. 
(In fact, $S$ has positive asymptotic upper density\, \cite{GK}).
\end{proposition}

We will use the following corollary of this result.
\begin{corollary} \label{GKC} 
Let $a_n^{(1)}(x),\; \dots , a_n^{(m)}(x)$ be  integrable subadditive cocycles with exponents 
$\la^{(1)}>-\infty, \;\dots, \;\la^{(m)}>-\infty$, respectively, 
over an ergodic measure-preserving system $(X,f,\mu)$.
Then there exists a  set  $E \subset X$ with $\mu(E)=1$ such that 
for each $x\in E$ and each $\e >0$ there exists an integer $L=L(x,\e)$ so that
the set $S=S(x,\e, L)$ of integers $n$ satisfying the following condition is infinite:
\begin{equation}\label{good n 2}
a_n^{(j)}(x)-a_{n-i}^{(j)} (f^ix) \ge (\la^{(j)}-\e)i  \quad
\text{for all }  L\le i \le n  \text{ and }1 \le j \le  m.
\end{equation}
\end{corollary}

\begin{proof}
We apply Proposition  \ref{GK}  to the subadditive cocycle 
$$
a_n(x)=a_n^{(1)}(x)+ \cdots + a_n^{(m)}(x)
$$ 
with exponent 
$\la =\la^{(1)}+\dots + \la^{(m)}>-\infty$ and obtain the set $E'$ for this cocycle. 
Then for each $x\in E'$ and each $\e >0$ there exists  $L'=L'(x,\e)$ and an infinite set 
$S=S(x,\e,L')$ such that for all $n\in S$ we have
\begin{equation}\label{good nm}
\sum_{j=1}^m \left( a_n^{(j)}(x)-a_{n-i}^{(j)} (f^ix) \right) \ge \left(\sum_{j=1}^m \la^{(j)}- \e\right)i 
\;\quad\text{for all } L'\le i \le n.
\end{equation}
By the Subadditive Ergodic Theorem there exists a set $G$ of full measure such that
 for  each $x\in G$ and each $j=1,\dots , m$ we have
$ 
\lim_{n \to \infty} n^{-1} a_n^{(j)}(x)= \la^{(j)}$
and hence there exists $M=M(x,\e)$ such that for 
$$
a_i^{(j)}(x) \le (\la^{(j)}+\e)i \quad\text{for all }\, i \ge M \text{ and } 1\le j \le m.
$$
It follows  by subadditivity that for all $j$  and all $n\ge M$
$$
  a_n^{(j)}(x)- a_{n-i}^{(j)} (f^ix)\le a_i^{(j)}(x) \le  (\la^{(j)}+\e)i \quad\text{for all }\, M \le i \le n.
  $$
We set $L=L(x,\e)=\max\{L',M\}$ and $E=G\cap E'$,  with $\mu(E)=1$.
Subtracting the inequalities above for $j=2, \dots , m$ from \eqref{good nm}
we get that for all $x \in E$ and $n\in S$ 
$$
 a_n^{(1)}(x)-a_{n-i}^{(1)} (f^ix) \ge (\la^{(1)}-m\e)i \quad\text{for all }\, i \text{ with } L\le i \le n.
 $$
The inequalities for the = cocycles $a_n^{(j)}$, $j=2, \dots, m$, follow similarly.
\end{proof}


\subsection{Lyapunov norm for upper and lower Lyapunov exponents} \label{Lyapunov norm}

Since the Multiplicative Ergodic Theorem does not apply in the infinite dimensional setting,
we use a cruder version of Lyapunov norm which takes into account only  upper and lower 
Lyapunov exponents. 
We fix  $\e >0$ and for any point $x\in \La$ we define the {\it Lyapunov norm}  of $u\in \V_x$
 as follows:
\begin{equation}  \label{Lprod}
\|u\|_{x} = \|u\|_{x,\e}  = \sum_{n=0}^\infty \, \|  \A_x^n \,u\| \,e^{-(\la_+ +\e) n}
+\sum_{n=1}^\infty \, \|  \A_x^{-n} \,u\| \,e^{(\la_- -\e) n}.
\end{equation}
By the definition \eqref{Lambda def} of $\La$, both series converge exponentially.
Properties of this norm were obtained in \cite[Proposition 3.1]{KS17}. They are similar
to those of the usual Lyapunov norm discussed in Section \ref{SLyapunovmetric}.
In particular, 
\begin{equation}  \label{estAnorm inf}
 \| \A_x^n \|_{x_n \leftarrow x_0} \le  e^{n (\la_{+\,}+ \e)}\quad\text{and}\quad
\| (\A_x^n)^{-1} \|_{x_0 \leftarrow x_n} \le  e^{n (-\la_{-\,}+ \e)} \quad\text{for }x\in \La
\end{equation}
and the ratio to the background norm $\|u\|_{x}/\|u\|$ is ``tempered". Specifically, there exist
an $f$-invariant set $\La' \subset \La$ with $\mu(\La')=1$ and
a measurable function $K_{\e}(x)$ on $\La'$ satisfying conditions 
\eqref{estK}  and \eqref{estLnorm}.
For any $\ell >1$ we  define 
\begin{equation}  \label{Pset inf}
\La_\ell = \{x \in \La' : \; \; K(x) \le \ell \},
\end{equation}
and note that $\mu (\La_\ell) \to 1$ as $\ell \to \infty$.

We use this Lyapunov norm to obtain estimates similar to \eqref{estAnorm inf}  
for {\em any} point $p\in X$ whose trajectory is close to that of a point $x\in \La_\ell$. 
Since the Lyapunov norm may not exist at points $f^ip$ we will use the Lyapunov 
norms at the  corresponding points  $f^ix$ for the estimates.

\begin{lemma} \cite[Lemma 4.1]{KS17} \label{mainest}
Let $f$ be an ergodic invertible measure-preserving transformation 
of a probability space $(X,\mu)$, let $\A$ be an  $\a$-H\"older 
cocycle over $f$ with the set $\{\A_x: x\in X\}$ bounded in $GL(V)$,
and let  $\la_+$ and $\la_-$ be the upper and lower Lyapunov exponents of $\A$. 

Then for any $\e>0$ and $\gamma$ with $\e < \g \a$ 
there exists a constant  $c=c\,(\A, \,\a \g - \e)$ such that 
for  any  point $x$  in $\La_{\e,\ell}$ with $f^k x$ in $\La_{\e,\ell}$
and any point $p \in X$ such that the orbit segments $x, fx, \dots , f^k x$ and 
$\,p, fp, \dots , f^k p$ satisfy with some $\delta>0$
\begin{equation}  \label{d-close}
\dist (f^i x, f^i p) \le \delta e^{ -\g\, \min\{i,\,k-i \} }
\quad\text{for every }\,i=0, \dots , k
\end{equation}
we have for every $\,i=0, \dots , k$
\begin{equation}  \label{d-close-coc}
  \| \A_p^i \| \le \ell\, \| \A_p^i \|_{x_i \leftarrow x_0} \le 
\ell \, e^{c \,  \ell \delta^\a} e^{i (\la_+ + \e)} \quad\text{ and}
\end{equation}
\begin{equation}  \label{d-close-coc-1}
  \| (\A_p^i)^{-1} \| \le \ell e^{\e \min\{i,\,k-i \}} \| (\A_p^i)^{-1} \|_{x_0 \leftarrow x_i} \le 
\ell e^{\e \min\{i,\,k-i \}} e^{c \,  \ell \delta^\a} e^{i (-\la_- + \e)}.
\end{equation}
\end{lemma}


\section{Proof of Theorem \ref{main inf}}
We fix a  sufficiently small $\e'>0$ and sufficiently large $\ell' \in \N$ so that $\mu(\r _{\e',\ell'})>0.9$
and so that we can apply the Closing Lemma \ref{closing} and get $\g =\g (\e',\ell') >0$. 
We define $\e _0 = \min \{  \a \g /4 ,\,  \e' \} >0$.  We fix $0<\e<\e_0$ and consider the sets 
$\La_{\ell, \e}$ for the cocycle $\A$. We choose $\ell$ sufficiently large so that $\mu (P)>0.8$ 
where 
$$
P=\La _{\ell, \e} \cap \r_{\ell', \e'}.
$$
We apply Corollary \ref{GKC} to the subadditive cocycles 
$$
a_n^{(1)}(x)=a_n(x)=\log \|\A^n_x\|\quad\text{and}\quad a_n^{(2)}(x) =\tilde a_n(x)=\log \|(\A^n_x)^{-1}\|
$$
and obtain the set $E$ of full measure 
for these two cocycles.
\vskip.1cm

We take a compact set $K \subset (P \cap E)$ with $\mu (K)> 0.7$, let 
$\nu$ be the restriction of $\mu$ to $K$, and denote by $G$ the support of $\nu$. Then $G$ is a compact 
subset of $K$ satisfying $\nu (X\setminus G) =0$, $\nu (G) =\mu (K)> 0.7$,\, and $\nu (U)>0$ for any subset $U$ (relatively) open in $G$. We fix a countable basis $\{ U_j\}$, $U_j \subset G$, 
 for the topology of $G$ and note  that  $\mu (U_j) = \nu (U_j)>0$. We denote by $G'$
 the subset of full $\mu$-measure in $G$ on which the Birkhoff Ergodic Theorem holds 
 for the indicator functions of each $U_j$, that is for each $x\in G'$
\begin{equation}  \label{bet}
\lim_{n \to \infty} \,n^{-1}\left| \{i:  0 \le i  \le n-1 \text{ and } f^i x \in U_j \} \right|
= \mu(U_j)> 0 \quad\text{for all }j\in \N.
\end{equation}

For the remainder of the proof we fix a point $x\in G'$  and choose $\sigma = 4\e/(\a\g)$.
We establish the following lemma, which is a refinement of \cite[Lemma 8]{G} for our setting.

\begin{lemma} \label{G}
For each $x \in G'$, $\beta >0$ and $\sigma>0$ there exists an integer 
 $N=N(x,\sigma,\b)$ such that  for each  $n \ge N$ there is an integer $k$ satisfying
 $$
 n(1 + \sigma) \le  k \le  n(1 + 2\sigma), \;\quad f^kx\in  G, 
\quad\text{and}\quad  \dist(x, f^k x) < \b.
$$
\end{lemma}

\begin{proof} We fix  $x \in G'$ and $\beta >0$ and consider a set $U_j$ from the countable base
 that  is contained in the open ball $B(x,\beta/2)$.  Since $x\in G'$, condition \eqref{bet} holds for $x$. 
  Denoting the set of the return times $i$  to $U_j$  by $T$ 
and letting $\phi(n)= |\,T \cap [0, n-1]\,|$ we get that $\phi(n) / n \to t=\mu (U_j)>0$.
 For any $\sigma>0$ we can take $c>0$ such that 
 $$
 (1+c)/(1-c)<(1 + 2\sigma) /(1 + \sigma)
 $$
  and then $M$ such that 
 $$
 (1-c) tn <\phi(n)  <(1+c) tn \quad\text{for all }n \ge M.
 $$
Using this and the fact that  $(1-c)(1 + 2\sigma)-(1+c)(1 + \sigma) >0$ by the choice of $c$, 
we obtain that there exists $N>M$ 
such that for all $n \ge N$ we have
$$
\phi(n(1 + 2\sigma))-\phi(n(1 + \sigma))> [(1-c)(1 + 2\sigma)-(1+c)(1 + \sigma)]\,tn >1,
$$
which means there exists $k$ between  $n(1 + \sigma)$ and  $ n(1 + 2\sigma)$ such that
$f^kx$ is in  $U_j \subset (G \cap B(x,\beta/2))$.
\end{proof}

Since $x\in G' \subset E$, by  Corollary  \ref{GKC} there exists an integer $L=L(x,\e)$ such that 
for any $n\in S_x$ and  any $i$ with $L\le i \le n$
\begin{equation}\label{good n3}
a_n(x)-a_{n-i} (f^ix) \ge (\la_+-\e)i \quad \text{ and } \quad
\tilde a_n(x)-  \tilde a_{n-i} (f^ix) \ge (-\la_- -\e)i .
\end{equation}

We conclude that for any $\beta>0$ there exist arbitrarily large $n$ for which
the above property holds and a corresponding $k=k(n)$ satisfying the conclusion of 
Lemma \ref{G}. 
We will later choose $\delta>0$ sufficiently small so that \eqref{Cdelta} below is satisfied,
and take $\b=\b(\delta, \ell, \g)>0$ from the Closing Lemma \ref{closing}, which then gives
 a periodic point $p=f^kp$ such that
 \begin{equation}  \label{x-p}
\dist (f^i x, f^i p) \le \delta e^{ -\g\, \min\{i,\,k-i \} }
\quad\text{for every }i=0, \dots , k.
\end{equation}

\vskip.2cm
 
\noindent {\bf Obtaining upper estimates for $\| \A^k_p\|$ and $\| (\A^k_p)^{-1}\|$.} $\,$
Since $x$ and $p$ satisfy \eqref{x-p}  we can apply
 Lemma \ref{mainest} with  $i=k$ to get
  $$
 \| \A_p^k \| \le \ell e^{c \,  \ell \delta^\a} e^{k (\la_{+\,} + \e)} 
\quad\text{and}\quad
 \| (\A_p^k)^{-1} \| \le \ell e^{c \,  \ell \delta^\a} e^{k (-\la_{-\,} + \e)}. 
$$
For $k$ sufficiently large compared to $\ell$ and $c=c(\A,\, \a\g-\e)$, we obtain 
\begin{equation}\label{upper}
k^{-1} \log \| \A_p^k \| \le 
\la_{+\,} + \e+ k^{-1}(\log \ell + c \,  \ell \delta^\a)  \le  \la_+ + 2\e, \quad \text{and}
\end{equation}
\begin{equation}\label{upper -}
k^{-1} \log \| (\A_p^k)^{-1} \| \le 
-\la_- + \e+ k^{-1}(\log \ell + c \,  \ell \delta^\a)  \le  -\la_- + 2\e.
\end{equation}

\vskip.2cm 

\noindent {\bf Obtaining a lower estimate for $\| \A^k_p\|$.} $\,$
First we bound $\|\A^n_x -\A_p^n \|$.
 $$
\begin{aligned}
\A^n_x -\A_p^n \,&=\, 
\,\A^{n-1}_{x_1}\circ (\A_x -  \A_p) + (\A^{n-1}_{x_1}- \A^{n-1}_{p_1})\circ \A_p \\
& =\, \A^{n-1}_{x_1}\circ (\A_x -  \A_p) + 
\left( \A^{n-2}_{x_2} \circ (\A_{x_1} - \A_{p_1}) + 
 (\A^{n-2}_{x_2} - \A^{n-2}_{p_2})\circ \A_{p_1}\right) \circ \A_p \\
& = \,  \A^{n-1}_{x_1}\circ (\A_x -  \A_p) + 
\A^{n-2}_{x_2} \circ (\A_{x_1} - \A_{p_1})  \circ \A_p + 
 (\A^{n-2}_{x_2} - \A^{n-2}_{y_2})\circ \A^2_p \\
 & = \dots = \, \sum_{i=0}^{n-1} \,\A^{n-(i+1)}_{x_{i+1}} \circ (\A_{x_i}-\A_{p_i}) \circ \A^i_p.
\end{aligned}
$$
Hence we can estimate the norm as follows
\begin{equation}\label{e1}
 \| \A^n_x -\A_p^n \| \le \sum_{i=0}^{n-1} 
 \,\|\A^{n-(i+1)}_{x_{i+1}}\| \cdot \|\A_{x_i}-\A_{p_i}\| \cdot \|\A^i_p\|.
\end{equation}
Since $n$ satisfies \eqref{good n3} with $a_n(x)=\log \| \A^n_x \|$,
\begin{equation}\label{e l} 
 \|\A^{n-(i+1)}_{x_{i+1}}\| \le \|\A^n_x\|\, e^{-(i+1)(\la_+ -\e)} \quad\text{for all $i$ with }L\le i \le n-1.
\end{equation}
Since $n < k$, applying Lemma \ref{mainest}  we get
\begin{equation}\label{e r}
 \|\A^i_p\| \le \ell \,e^{c \,  \ell \delta^\a} e^{i (\la_{+\,} + \e)}\quad\text{ for  }i =0,\dots, n.
\end{equation}
Using \eqref{x-p} and H\"older continuity of $\A$ we obtain   
$$
  \|\A_{x_i}-\A_{p_i}\| \le M\,\dist(x_i,p_i)^\a  
  \le M( \delta \, e^{ -\g\, \min\{i,k-i \}} ) ^\a  = M \delta^\a \, e^{ -\a \g\, \min\{i,\,k-i \}}. 
  $$
We claim that the exponent satisfies
\begin{equation}\label{exponent}
\a \g \min\{i,k-i \} \ge 4\e i \quad \text{for } \;i=0,\dots ,n.
\end{equation}
If $i=\min\{i,k-i \}$ this holds since $\e<\a \g/4$.\, 
 If $k-i=\min\{i,k-i \}$ then 
$$
\a \g \min\{i,k-i \} =\a \g(k-i) \ge 4\e i \quad\text{is equivalent to}\quad i \le  k/(1+4\e/(\a\g)),
$$
which holds for $i \le n$ since $n\le k/(1+\sigma)$ and $\sigma = 4\e/(\a\g)$.
Thus we conclude that
\begin{equation}\label{decay}
  \|\A_{x_i}-\A_{p_i}\| \le M \delta^\a \, e^{ -4\e i}  \quad \text{for } i=0, \dots, n.
\end{equation}
Combining \eqref{e l}, \eqref{decay}, and \eqref{e r} we obtain that for $L \le i \le n-1$
\begin{equation}\label{term}
\begin{aligned}
&\|\A^{n-(i+1)}_{x_{i+1}}\| \cdot \|\A_{x_i}-\A_{p_i}\| \cdot \|\A^i_p\| \le\\
&\le  \|\A^n_x\|\, e^{-(i+1)(\la_{+\,} -\e)} \cdot M \delta^\a \, e^{ -4\e i} \cdot
 \ell \,e^{c \,  \ell \delta^\a} e^{i(\la_{+\,} + \e)} <  C_1(\delta)\, \|\A^n_x\|\,  e^{ -\e i},
\end{aligned}
\end{equation}
where $\,C_1(\delta)=\delta^\a M \ell\, e^{c \,  \ell \delta^\a -\la_{+\,}+\e}$,\,
and we conclude that 
\begin{equation}\label{e2}
\begin{aligned}
& \sum_{i=L}^{n-1} 
\,\|\A^{n-(i+1)}_{x_{i+1}}\| \cdot \|\A_{x_i}-\A_{p_i}\| \cdot \|\A^i_p\|  \le \\
& \le \, C_1(\delta) \,\|\A^n_x\|\, \sum_{i=L}^{n-1}  e^{-\e i} 
\,\le\,  C_1(\delta) \,\|\A^n_x\|\, \frac{1}{1-e^{-\e} } \,=\, C_2(\delta)\, \|\A^n_x\|.
\end{aligned}
\end{equation}

\vskip.1cm

Since the set $\{\A_x: x\in X\}$ is bounded in $GL(V)$, there exists $ \la_*\le  \la_- \le \la_+$
such that $\| (\A_x)^{-1}\| \le e^{-\la_*}$ for all $x\in X$. For $i<L$ we estimate 
$$
\|\A^{n-i}_{x_{i}}\|  \le \|\A^n_x\| \cdot \| (\A_{x}^{i})^{-1}\| \le \|\A^n_x\|\, e^{-\la_* i}.
$$
Then using \eqref{e r} and \eqref{decay} we obtain
\begin{equation}\label{e3}
\begin{aligned}
& \sum_{i=0}^{L-1}  
\|\A^{n-(i+1)}_{x_{i+1}}\| \cdot \|\A_{x_i}-\A_{p_i}\| \cdot \|\A^i_p\| \le \\
&\le \, \sum_{i=0}^{L-1} \|\A^n_x\|\, e^{-(i+1)\la_*} \cdot M\delta^\a e^{-4\e i} \cdot
 \ell e^{c \,  \ell \delta^\a} e^{i(\la_{+\,} + \e)}  \le  \\
 &\le L \cdot \delta^\a M \ell e^{c \,  \ell \delta^\a }\, e^{-\la_*+(\la_{+\,}-\la_*)L}  \cdot
  \|\A^n_x\| = C_3(\delta) \, \|\A^n_x\|.
\end{aligned}
\end{equation}

Combining  estimates \eqref{e1},  \eqref{e2} and  \eqref{e3} we obtain 
 $$
  \| \A^n_x -\A_p^n \| \,\le\, \|\A^n_x\|\, (C_2(\delta)   + C_3(\delta))
  \,\le \, \|\A^n_x\|/2
 $$
 since by the choice of $\delta>0$ we have
 \begin{equation}\label{Cdelta}
 C_2(\delta)   + C_3(\delta) =  
 \delta^\a M \ell e^{c \,  \ell \delta^\a} \left((1-e^{-\e})^{-1}e^{-\la_{+\,} +\e}+Le^{-\la_*+(\la_{+\,}-\la_*)L}\right) < 1/2.
\end{equation}
Hence 
$$
 \|\A^n_p\| \,\ge\, \|\A^n_p\| - \| \A^n_x -\A_p^n \| \,\ge\, \|\A^n_x\|/2 
 > e^{(\la_{+\,} -\e)n}/2,
$$
provided that $n$ is sufficiently large for the limit in \eqref{la}. Since
$\A^n_p=(\A^{k-n}_{f^np})^{-1} \circ \A^k_p$,  
$$
\| \A^n_p\| \le \| (\A^{k-n}_{f^np})^{-1}\| \cdot \|\A^k_p\| \le \|\A^k_p\| \, e^{-\la_* (k-n)}.
$$
Hence 
$$
\|\A^k_p\|  \ge e^{\la_* (k-n)} \| \A^n_p\|   >  e^{(\la_{+\,} -\e)n+\la_* (k-n)}/2 >
e^{(\la_{+\,} -\e)k-(\la_{+\,}-\la_*) (k-n)}/2, \, \text{ and so} 
$$
$$
k^{-1} \log \|\A^k_p\| > k^{-1} [(\la_{+\,} -\e)k-(\la_{+\,} -\la_*) (k-n) - \log 2].
$$
Since $k-n< 2\sigma n <  2\sigma k$ and $\sigma=4\e/(\a \g)$ we obtain
$$
k^{-1} \log \|\A^k_p\| >  \la_{+\,} -\e- k^{-1} \log 2 -(\la_{+\,} -\la_*) 2\sigma > \la_{+\,} -2\e - 8\e(\la_{+\,} -\la_*)/(\a\g)
$$
if $n$ and hence $k$ are sufficiently large. 

Since $0<\e<\e_0$ is arbitrary and
$\la_+ ,\, \la_*, \a, \g$ do not depend on $\e$, this inequality above together with \eqref{upper}
imply the approximation of $\lambda_+(\A,\mu)$ by $k^{-1} \log \|\A^k_p\|$.

\vskip.3cm

\noindent {\bf Obtaining a lower estimate for $\| (\A^k_p)^{-1}\|$.} $\,$
We use the equation
$$
 (\A^n_x)^{-1} -(\A_p^n)^{-1}  \,=\, 
  \sum_{i=0}^{n-1} \,(\A^i_p)^{-1}  \circ \left( (\A_{x_i})^{-1}-(\A_{p_i})^{-1} \right) 
  \circ (\A^{n-(i+1)}_{x_{i+1}})^{-1}
$$
and estimate the norm of the difference similarly to \eqref{e1}: 
\begin{equation}\label{e1-}
 \|(\A^n_x)^{-1} -(\A_p^n)^{-1}\|  \,=\, 
  \sum_{i=0}^{n-1} \,\|(\A^i_p)^{-1}\|  \cdot \| (\A_{x_i})^{-1}-(\A_{p_i})^{-1} \| 
  \cdot \|(\A^{n-(i+1)}_{x_{i+1}})^{-1}\|.
\end{equation}
Since $n$ satisfies the second part of \eqref{good n3} with $\tilde a_n(x)=\log \| (\A^n_x)^{-1} \|$,
 $$ 
 \| (\A^{n-(i+1)}_{x_{i+1}})^{-1}\| \le \|(\A^n_x)^{-1}\|\, e^{-(i+1)(-\la_- -\e)} \quad\text{for all $i$ with }L\le i \le n.
 $$
  Since $n < k$, using \eqref{d-close-coc-1} of Lemma \ref{mainest}   we get
 $$
 \|(\A_{p}^i)^{-1}\| \le \ell e^{\e \min\{i,\,k-i \}} \,e^{c \,  \ell \delta^\a} e^{i (-\la_{-\,} + \e)}
 \le \ell \,e^{c \,  \ell \delta^\a} e^{i (-\la_{-\,} + 2\e)} \quad\text{ for }i =1,\dots, n.
 $$
Using H\"older continuity and the exponent estimate \eqref{exponent} we obtain as in \eqref{decay} that
$$
  \|(\A_{x_i})^{-1}-(\A_{p_i})^{-1}\| \le M\,\dist(x_i,p_i)^\a  
    \le M \delta^\a \, e^{ -\a \g\, \min\{i,\,k-i \}}  \le M \delta^\a \, e^{ -4\e i}
 $$
 for  $i=1, \dots, n$. 
Combining these estimates we obtain for the terms in \eqref{e1-} the same estimate as in \eqref{term}
\begin{equation}\label{term-}
\|(\A^i_p)^{-1}\|  \cdot \| (\A_{x_i})^{-1}-(\A_{p_i})^{-1} \| 
  \cdot \|(\A^{n-(i+1)}_{x_{i+1}})^{-1}\| \le
 C_1(\delta)\, \|(\A^n_x)^{-1}\|\, e^{-\e i}
\end{equation}
for all $i$  with $L \le i \le n-1$.
The remainder of the argument is essentially identical to estimating $\|\A^k_p\|$
with $\A$ replaced by $\A^{-1}$.
 
 \vskip.1 cm
 
 It is clear from the argument that we can choose arbitrarily large $n$ and hence $k$.
 
 \vskip.2cm
 
\noindent {\bf Simultaneous approximation for several cocycles.}$\,$
Finally, we remark on how to obtain simultaneous approximation for cocycles
$\A^{(1)}, \dots , \A^{(m)}$ in Remark \ref{several}. First, we take the set  $\La_{\e,\ell}$
to be the intersection of the corresponding sets for all $\A^{(j)}$ and define the set 
$P=\La_{\e,\ell} \cap \r_{\e',\ell'}$  as in the proof. 
We also take $E$ to be the full measure set given by Corollary \ref{GKC} for $2m$ subadditive
cocycles 
$$
a^{(j)}(x)=\log \| \A^{(j)}\|, \quad \tilde a^{(j)}(x)=\log \| (\A^{(j)})^{-1}\|, \quad j=1, \dots, m.
$$
Then the argument goes through showing that the constructed point $p=f^kp$
 gives the approximation of the upper and lower exponents for all cocycles $\A^{(1)}, \dots , \A^{(m)}$.



\end{document}